\documentclass{amsart}

\usepackage{amsmath,amssymb,amsthm}
\usepackage{pinlabel}

\usepackage{units} 

\newtheorem{prop}{Proposition}[section]
\newtheorem{thm}[prop]{Theorem}
\newtheorem{lem}[prop]{Lemma}

\theoremstyle{definition}

\newtheorem{defn}[prop]{Definition}
\newtheorem{ex}[prop]{Example}
\newtheorem{rem}[prop]{Remark}

\newtheorem*{ack}{Acknowledgements}

\def\co{\colon\thinspace}

\newcommand{\tb}{\mathtt{tb}}
\newcommand{\otb}{\overline{\mathtt{tb}}}
\newcommand{\lk}{\mathtt{lk}}
\newcommand{\vlk}{\underline{\lk}}
\newcommand{\rot}{\mathtt{rot}}
\newcommand{\vrot}{\underline{\rot}}
\newcommand{\ttt}{\mathtt{t}}

\newcommand{\rma}{\mathrm{a}}
\newcommand{\rmb}{\mathrm{b}}

\newcommand{\rme}{\mathrm{e}}
\newcommand{\rmi}{\mathrm{i}}

\newcommand{\bfx}{\mathbf{x}}
\newcommand{\bfy}{\mathbf{y}}

\newcommand{\C}{\mathbb{C}}
\newcommand{\N}{\mathbb{N}}
\newcommand{\Q}{\mathbb{Q}}
\newcommand{\R}{\mathbb{R}}
\newcommand{\Z}{\mathbb{Z}}

\newcommand{\LL}{\mathbb{L}}

\newcommand{\xist}{\xi_{\mathrm{st}}}

\newcommand{\lambdac}{\lambda_{\mathrm{c}}}
\newcommand{\lambdas}{\lambda_{\mathrm{s}}}
\newcommand{\loss}{\widehat{\mathfrak{L}}}

\DeclareMathOperator{\Int}{Int}

\begin{document}

\author[H.~Geiges]{Hansj\"org Geiges}
\address{Mathematisches Institut, Universit\"at zu K\"oln,
Weyertal 86--90, 50931 K\"oln, Germany}
\email{geiges@math.uni-koeln.de}

\author[S.~Onaran]{Sinem Onaran}
\address{Department of Mathematics, Hacettepe University,
06800 Beytepe-Ankara, Turkey}
\email{sonaran@hacettepe.edu.tr}

\title{Exceptional Legendrian torus knots}

\date{}

\begin{abstract}
We present classification results for exceptional Legendrian realisations
of torus knots. These are the first results of that kind for
non-trivial topological knot types.
Enumeration results of Ding--Li--Zhang
concerning tight contact structures on certain Seifert fibred manifolds with
boundary allow us to place upper bounds on the number of tight
contact structures on the complements of torus knots; the
classification of exceptional realisations of these torus knots
is then achieved by exhibiting sufficiently many realisations
in terms of contact surgery diagrams. We also
discuss a couple of general theorems about the existence
of exceptional Legendrian knots.
\end{abstract}



\maketitle


\section{Introduction}
The classification of Legendrian knots is one of the basic
questions in $3$-dimen\-sio\-nal contact topology. The first
classification result --- for Legendrian
realisations of the topological unknot in the $3$-sphere $S^3$
with its standard tight contact structure~$\xist$ ---
is due to Eliashberg and Fraser~\cite{elfr98}; see also~\cite{elfr09}.
Legendrian realisations of torus knots and the figure eight knot
in $(S^3,\xist)$ were classified by Etnyre and Honda~\cite{etho01}.
In those topological knot types one can determine the range of the classical
invariants $\tb$ (Thurston--Bennequin invariant) and $\rot$
(rotation number), and it is shown that these invariants suffice
to distinguish the knots up to Legendrian isotopy.

In general, further invariants are required for a complete
classification. The first example of that kind was discovered by
Chekanov~\cite{chek02} and Eliashberg: there are two Legendrian realisations
of the $5_2$ knot with the same classical invariants that can
be distinguished by a differential graded algebra associated with
the Legendrian knot.

The Legendrian classification question can be extended in various directions:
Legendrian links~\cite{dige07}, knots or links in other tight contact
$3$-manifolds~\cite{baet12,cdl15,dige10,ghig06,onar17}, or knots in
overtwisted contact structures~\cite{elfr09,etny13,geon15}.
Here we have only cited a few examples where again the classical invariants
suffice to distinguish all Legendrian realisations.

Knots of the latter kind are the object of interest in this
paper. They fall into two classes.

\begin{defn}
A Legendrian knot $L$ in an overtwisted contact $3$-manifold $(M,\xi)$
is called \emph{exceptional} (or \emph{non-loose}) if its complement
$(M\setminus L,\xi|_{M\setminus L})$ is tight; $L$
is called \emph{loose} if the contact structure is still
overtwisted when restricted to the knot complement.
\end{defn}

As shown by Etnyre~\cite[Theorem~1.4]{etny13}, loose Legendrian knots in
null-homolo\-gous knot types in any contact $3$-manifold are classified by
the classical invariants (and the range of the invariants is only restricted
by $\tb+\rot$ being odd). See also \cite[Theorem~0.1]{dyma01}
and \cite[Theorem~6]{dige09}.

\begin{rem}
Here by `classification' we always mean the classification of
\emph{oriented} Legendrian knots up to \emph{coarse equivalence},
i.e.\ up to contactomorphism of the ambient manifold.
Since the contactomorphism groups of contact manifolds
other than $(S^3,\xist)$ are not, in general, connected,
the classification up to Legendrian isotopy is more subtle.
\end{rem}

The classification of exceptional Legendrian knots is more involved
than that of loose ones.
Exceptional topological unknots were classified by Eliashberg and
Fraser~\cite{elfr09}; see \cite[Theorem~5.1]{geon15} for an
alternative argument.
Recall that, up to isotopy, there is an integer family of overtwisted
contact structures on~$S^3$, which are distinguished by the Hopf invariant
$h\in\Z$ of the underlying tangent $2$-plane field. We work
instead with the $d_3$-invariant (see \cite{dgs04} for its general definition);
for $S^3$ the two invariants are related by $d_3=-h-1/2$.
We write $\xi_d$ for the overtwisted contact structure on $S^3$
characterised by $d_3(\xi_d)=d\in\Z+1/2$.
The standard contact structure on $S^3$ has $d_3(\xist)=-1/2$; thus,
whenever we have a contact structure on $S^3$ (e.g.\ a structure
obtained by one of the surgery
diagrams used in this paper) with $d_3\neq-1/2$, we know right away that
the contact structure is overtwisted.

\begin{thm}[Eliashberg--Fraser]
\label{thm:EF}
Let $L\subset(S^3,\xi)$ be an exceptional unknot in an overtwisted contact
structure $\xi$ on the $3$-sphere. Then
$\xi=\xi_{1/2}$, and
the classical invariants can take the values
\[ \bigl(\tb(L),\rot(L)\bigr)\in\bigl\{(n,\pm(n-1))\co n\in\N\bigr\}.\]
These invariants determine $L$ up to coarse equivalence.
\end{thm}

In \cite{geon15} we classified exceptional rational
unknots in lens spaces. In the present paper we achieve the first
classification of exceptional realisations of non-trivial
topological knot types.

One issue that did not arise in \cite{elfr09} or \cite{geon15} was that of
positive Giroux torsion (as defined in~\cite{giro94}),
because there the knot complement was a solid torus $S^1\times D^2$,
which would become overtwisted by introducing Giroux torsion along
the boundary, i.e.\ a Giroux torsion domain $T^2\times[0,1]$ with
$T^2\times\{t\}$ isotopic to the boundary torus $S^1\times\partial D^2$,
cf.~\cite[p.~68]{cdl15}. For non-trivial topological knot types, exceptional
Legendrian realisations may well have positive Giroux torsion in the
complement.

As shown by Etnyre~\cite{etny13}, see also the discussion
in~\cite{stve09}, one can produce infinitely many exceptional Legendrian
knots --- with the same classical invariants --- that are not coarsely
equivalent, by introducing Giroux torsion along an incompressible torus
in the knot complement, which does not change the ambient contact structure.

When one wants to classify exceptional Legendrian realisations
in a given topological knot type, it is therefore reasonable to impose the
stronger condition that the knot complement is not only tight,
but that it has zero Giroux torsion. Recall that
overtwisted contact structures have infinite Giroux torsion.
We adopt the following terminology from~\cite{stve09}.

\begin{defn}
A Legendrian knot $L$ in an overtwisted contact $3$-manifold $(M,\xi)$
is called \emph{strongly exceptional} if its complement
$(M\setminus L,\xi|_{M\setminus L})$ has zero Giroux torsion.
\end{defn}

Even in the strongly exceptional case,
the classical invariants may not distinguish all Legendrian realisations,
as shown in~\cite{stve09}.

Our aim in this paper will be to classify strongly exceptional
realisations of certain torus knots. We begin in
Section~\ref{section:exceptional-exists} with the observation
that, as a consequence of the Legendrian surgery presentation
theorem~\cite{dige04,dgs04}, any closed, overtwisted contact $3$-manifold
contains an exceptional Legendrian knot. (This also follows from the
work of Etnyre and Vela-Vick~\cite{etvv10}.) In
Section~\ref{section:exceptional-1/2} we give a sufficient
criterion for a topological knot type to admit an exceptional
realisation.

The strategy for classifying exceptional torus knots is similar to
the one we employed in~\cite{geon15} for the classification
of exceptional rational unknots. The complement of a torus knot
is a bounded $3$-manifold that admits a Seifert fibration
over the disc with two multiple fibres. For certain
boundary conditions, the tight contact structures (of zero Giroux
torsion) on such manifolds
have been classified by Ding--Li--Zhang~\cite{dlz13}. Their results give us
an upper bound on the number of Legendrian realisations whose complement
has zero Giroux torsion, in other words, strongly exceptional
realisations or realisations in $(S^3,\xist)$.
Realisations of the latter kind have been classified, as mentioned before,
by Etnyre--Honda~\cite{etho01}. It then remains to exhibit sufficiently
many strongly exceptional realisations
in terms of surgery diagrams. The new ingredient for
establishing that the examples are indeed strongly exceptional
is the LOSS invariant $\loss$ from~\cite{loss09}
and a vanishing theorem for $\loss$ due to
Stipsicz and V\'ertesi~\cite{stve09}.

In Section~\ref{section:exceptional-torus} we describe the Seifert
fibration of torus knot complements and summarise the relevant
results from~\cite{dlz13}. In Section~\ref{section:lht} we then
apply this to the left-handed trefoil knot; here the results of
\cite{dlz13} allow the most comprehensive classification.

\begin{thm}
\label{thm:lht}
The number of strongly exceptional Legendrian realisations $L$ of the
left-handed trefoil knot in $S^3$ is as follows:
\begin{itemize}
\item[(a)] For $\tb(L)=-5$ and $\tb(L)<-6$, there are precisely
two such realisations; they live in $(S^3,\xi_{3/2})$.
\item[(b)] For $\tb(L)=1$, there is at least one realisation.
\item[(c)] For each other value of $\tb(L)\in\Z$, there are
at least two realisations.
\end{itemize}
\end{thm}

All the known exceptional realisations in cases (b) and (c) likewise
live in the contact structure $\xi_{3/2}$ on~$S^3$. We conjecture that
the examples we shall describe constitute
a complete list of strongly exceptional realisations
of the left-handed trefoil knot.

In Section~\ref{section:rht} we discuss the classification of
strongly exceptional right-handed trefoils. Finally,
in Section~\ref{section:general} we give some classification results
for two general classes of torus knots. Some of the detailed
calculations of the classical invariants and the
$d_3$-invariant are relegated to Section~\ref{section:computations}.

\begin{rem}
As we shall see presently, all closed overtwisted contact $3$-manifolds
contain exceptional Legendrian knots. The two theorems we stated so far
indicate that
on a given differential manifold (here: the $3$-sphere), exceptional
realisations of a given knot type may exist in only one or very few
overtwisted contact structures. The examples below will confirm
this observation. This suggests that \emph{the `most simple' exceptional knot
in an overtwisted contact structure is a measure for its
complexity.}
\end{rem}
\section{Existence of exceptional knots}
\label{section:exceptional-exists}
The first example of an exceptional Legendrian knot was found
by Dymara~\cite{dyma01}. The intricacy of her construction
could lead one to expect such knots to be scarce.
The following theorem, however, shows that
every overtwisted contact manifold contains an exceptional knot.

\begin{thm}
\label{thm:exceptional-exists}
Any closed overtwisted contact $3$-manifold contains an exceptional Legendrian
knot.
\end{thm}

\begin{rem}
One way to prove this theorem is via the theory of open books
adapted to contact structures. According to \cite[Theorem~1.2]{etvv10},
the binding $B$ of any open book decomposition of $M$ supporting $\xi$
has zero Giroux torsion in its complement, i.e.\ $B$ is a strongly
exceptional transverse knot. By \cite[Proposition~1.2]{etny13},
any Legendrian approximation of $B$ is then likewise strongly
exceptional. (That proposition is formulated for exceptional
rather than strongly exceptional knots, but the proof also works
in the strongly exceptional case.)

Here we give a surgical proof that develops the idea on which
the explicit realisations of exceptional knots in the present paper
will be based.
\end{rem}

\begin{proof}[Proof of Theorem~\ref{thm:exceptional-exists}]
According to the surgery presentation theorem of~\cite{dige04},
any closed contact $3$-manifold $(M,\xi)$
can be obtained by contact $(\pm 1)$-surgery
on a suitable Legendrian link $\LL$ in $(S^3,\xist)$. As observed in
\cite[Corollary~1.4]{dgs04}, one may choose the Legendrian link
$\LL$ such that a contact $(+1)$-surgery on a single component $L_0$ of
$\LL$ and contact $(-1)$-surgeries on all other components
produces the desired manifold $(M,\xi)$.

The Legendrian push-off $L$ of $L_0$ in $(S^3,\xist)$ may be
regarded as a Legendrian knot in the surgered manifold $(M,\xi)$.
By the cancellation lemma of~\cite{dige04}, cf.\
\cite[Proposition~6.4.5]{geig08}, contact $(-1)$-surgery
on $L$ cancels the contact $(+1)$-surgery on~$L_0$. In other words,
contact $(-1)$-surgery on $L\subset (M,\xi)$ produces the same
contact manifold as contact $(-1)$-surgeries on $(\LL\setminus L_0)
\subset (S^3,\xist)$. Contact $(-1)$-surgery is symplectic handlebody
surgery \cite[Section~6.2]{geig08}, thus, the latter manifold is
symplectically fillable and hence tight. In particular, the complement of
$L$ in $(M,\xi)$ must have been tight.
\end{proof}

\begin{rem}
The proof of \cite[Corollary 1.4]{dgs04} implicitly relies
on the theory of open books adapted to contact structures.
For a proof entirely in the surgical realm see
\cite[Theorem~3.4]{geze13}.
\end{rem}

The following proposition says that exceptional knots realised
as in the proof of Theorem~\ref{thm:exceptional-exists}
will always be \emph{strongly} exceptional.

\begin{prop}
\label{prop:strongly}
Let $(M,\xi)$ be an overtwisted contact $3$-manifold
represented by a contact $(\pm 1)$-surgery diagram
containing a single $(+1)$-surgery. Then the
Legendrian knot $L$ in $(M,\xi)$ represented by the push-off
of the $(+1)$-surgery curve is strongly exceptional.
\end{prop}

\begin{proof}
In the foregoing proof we have seen that
contact $(-1)$-surgery along $L$ produces a strongly symplectically
fillable contact $3$-manifold. As shown by Gay~\cite[Corollary~3]{gay06},
positive Giroux torsion obstructs strong fillability. Thus,
contact $(-1)$-surgery
along $L$ produces a manifold with zero Giroux torsion. In particular,
the complement of $L$ in $(M,\xi)$ must have been of zero Giroux torsion.
\end{proof}
\section{Exceptional knots in $(S^3,\xi_{1/2})$}
\label{section:exceptional-1/2}
By Theorem~\ref{thm:EF}, the contact structure $\xi_{1/2}$ is
distinguished as the only overtwisted contact structure on
$S^3$ containing exceptional realisations of the topological unknot.
The following proposition gives
a sufficient criterion for a topological knot type to admit an exceptional
Legendrian realisation in $(S^3,\xi_{1/2})$.

\begin{prop}
Let $L$ be a Legendrian knot in $(S^3,\xist)$. If contact $(+1)$-surgery
on $L$ produces a tight contact $3$-manifold, then the topological knot
type of $L$ admits an exceptional realisation in $(S^3,\xi_{1/2})$.
\end{prop}

\begin{proof}
Given $L\subset (S^3,\xist)$, perform two contact $(+1)$-surgeries
along a standard Legendrian meridian and its push-off as
shown in Figure~\ref{figure:except-realise1}. The Kirby moves in
Figure~\ref{figure:except-realise2}
show that the surgered manifold is again the $3$-sphere, and the topological
knot type of $L$ is not affected by the surgeries. From the formula in
\cite[Corollary~3.6]{dgs04} (see equation~(\ref{eqn:d_3}) below)
it follows that the $d_3$-invariant of
the surgered contact structure equals $1/2$.

\begin{figure}[h]
\labellist
\small\hair 2pt
\pinlabel $L$ at 34 71
\pinlabel $+1$ [l] at 65 28
\pinlabel $+1$ [l] at 65 19
\endlabellist
\centering
\includegraphics[scale=1.4]{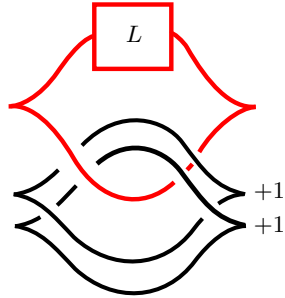}
  \caption{A Legendrian knot $L$ in $(S^3,\xi_{1/2})$.}
  \label{figure:except-realise1}
\end{figure}

\begin{figure}[h]
\labellist
\small\hair 2pt
\pinlabel $L$ at 8 74
\pinlabel $L$ at 114 74
\pinlabel $L$ at 29 18
\pinlabel $L$ at 113 18
\pinlabel $-1$ at 45 73
\pinlabel $0$ [tr] at 68 82
\pinlabel $0$ [bl] at 75 84
\pinlabel $1$ [b] at 158 79
\pinlabel $1$ [tl] at 165 64
\pinlabel $1$ [tl] at 178 69
\pinlabel $-1$ [bl] at 76 21
\endlabellist
\centering
\includegraphics[scale=1.8]{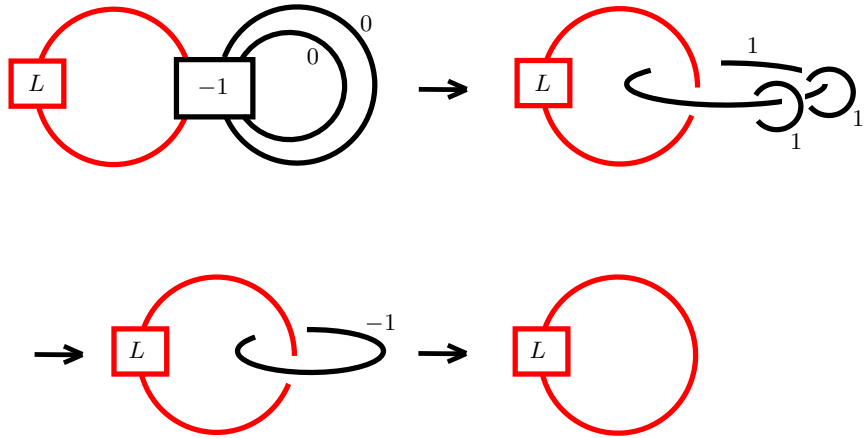}
  \caption{The topological type of $(S^3,L)$ is unchanged by the surgery.}
  \label{figure:except-realise2}
\end{figure}

In order to see that $L\subset (S^3,\xi_{1/2})$ is exceptional,
we perform contact $(-1)$-surgery on $L$ (in addition to the
two $(+1)$-surgeries). By \cite[Proposition~2]{dige09},
in the contact manifold obtained (from any initial
contact manifold) by a contact $(-1)$-surgery along a Legendrian knot~$K$,
the standard Legendrian meridian of $K$ is Legendrian isotopic
to the Legendrian push-off of~$K$. Thus, surgery along the three
knots as described is equivalent to a $(-1)$-surgery along $L$
and $(+1)$-surgeries along two push-offs of $L$. By the
cancellation lemma, this amounts to a single $(+1)$-surgery on~$L$
in $(S^3,\xist)$.

In conclusion, if contact $(+1)$-surgery on $L\subset (S^3,\xist)$ produces a
tight contact $3$-manifold, then so does $(-1)$-surgery
on $L\subset (S^3,\xi_{1/2})$, which proves this latter realisation
of $L$ to be exceptional.
\end{proof}

\begin{ex}
The basic example for this proposition is provided by the topological
unknot. Contact $(+1)$-surgery along its Legendrian realisation in
$(S^3,\xist)$ with $\tb=-1$ produces the tight contact structure
on $S^1\times S^2$.

Further examples are supplied by a theorem of Lisca
and Stipsicz. Let $K$ be a knot in $S^3$ with positive slice
genus $g_s>0$ and maximal Thurston--Bennequin invariant $\otb$
(of Legendrian realisations in $(S^3,\xist)$) equal to $2g_s-1$.
Then by \cite[Theorem~1.1]{list04} and its proof, contact $r$-surgery
with $r\neq 0$ along a Legendrian realisation $L$ of $K$ with
$\tb(L)=\otb$ yields a tight contact structure. This result applies,
for instance, to all positive $(p,q)$-torus knots, $p,q\geq 2$,
whose slice genus equals  $g_s=(p-1)(q-1)/2$, and whose
maximal Thurston--Bennequin invariant is $\otb=pq-p-q=2g_s-1$
by \cite[Theorem~4.1]{etho01}.

In a different context, these Legendrian torus knots were studied in
\cite[Example~4.1.13]{baon15}.
\end{ex}

\section{Exceptional torus knots}
\label{section:exceptional-torus}
In this section we provide some background
for the study of exceptional realisations of torus knots.
\subsection{Seifert fibrations of torus knot complements}
\label{subsection:Seifert}
For given coprime integers $p,q$, consider the $S^1$-action on
$S^3\subset\C^2$ given by
\[ \theta (z_1,z_2)=(\rme^{\rmi p\theta}z_1,\rme^{\rmi q\theta}z_2),
\;\;\;\theta\in S^1=\R/2\pi\Z.\]
This defines a Seifert fibration with two singular fibres
of multiplicity $|p|$ and $|q|$ through the points
$(1,0)$ and $(0,1)$, respectively. All regular fibres are copies
of the $(p,q)$-torus knot.

Choose integers $p',q'$ such that $pq'+p'q=1$. Then, with
the conventions of \cite[Section~2.2]{gela18},
the Seifert invariants are given by
\[ \bigl(g=0;(1,0),(p,p'),(q,q')\bigr),\]
where we include a Seifert pair $(1,0)$, corresponding to
a non-singular fibre, to represent a copy $K$
of the $(p,q)$-torus knot. In particular, the complement of $K$
is Seifert fibred. Notice that by \cite{mose71}, the property of
having a Seifert fibred complement characterises torus knots.

Let $\nu K$ be a closed tubular neighbourhood of $K$ made up of Seifert
fibres. The Seifert fibration on $S^3\setminus\Int(\nu K)$
has base $D^2$ and two singular fibres. 
In terms of the invariants used in \cite[Section~2]{dlz13},
this Seifert fibration is $M(D^2;-p'/p,-q'/q)$; this follows
by comparing the conventions there with those of
\cite[Section~2.2]{gela18}.

Write $\mu$ for the meridian on the boundary $\partial(\nu K)$ of the
tubular neighbourhood; as a longitude $\lambda$ we choose a
parallel Seifert fibre, so that the linking number between
$K$ and $\lambda$ is $\lk(K,\lambda)=pq$.
On the complement $M(D^2;-p'/p,-q'/q)$, the meridian $\mu$
is identified with $-\partial D^2\times\{1\}$; the longitude $\lambda$
corresponds to an $S^1$-fibre $\{*\}\times S^1$ for some
point $*\in\partial D^2$.

In \cite{dlz13}, the authors determine the number of tight
contact structures on the Seifert fibred manifold
$M(D^2;-p'/p,-q'/q)$ with minimal convex boundary
of slope~$s$ and zero Giroux torsion along the boundary, for a certain
range of permissible slopes. `Minimality' of the boundary
means that there are two dividing curves. The slope of the
dividing curves is measured with
respect to the basis $(\mu,\lambda)$ (under the above identifications).

Now let $L\subset (S^3,\xi)$ be a Legendrian realisation of
the $(p,q)$-torus knot $K$ in some contact structure $\xi$ on
the $3$-sphere. The boundary of a standard tubular neighbourhood
of $L$ is minimal convex. Write $\lambdas,\lambdac$ for the curves
on $\partial(\nu L)$ representing the surface framing (i.e.\
the framing provided by a Seifert surface) and the
contact framing, respectively. From $\lk(K,\lambda)=pq$ we have
\[ \lambdas=\lambda-pq\mu.\]
Hence
\[ \lambdac=\lambdas+\tb(L)\mu=\lambda+(\tb(L)-pq)\mu.\]
So the slope $s$ is given by
\[ s=\frac{1}{\tb(L)-pq}.\]
\subsection{Two families of torus knots}
We now consider two families of torus knots, where, as we shall see,
the classification results of Ding--Li--Zhang~\cite{dlz13} about tight
contact structures on certain Seifert fibred manifolds with
torus boundary apply,
and exceptional realisations can be described in surgery diagrams.
\subsubsection{Positive torus knots}
\label{subsubsection:positive}
We first look at positive $(p,q)$-torus knots, where $p$ is a natural number
greater than or equal to~$2$, and $q=np+1$ for some $n\in\N$.
This means that $p\cdot (-n)+1\cdot q=1$, so in
the notation of Section~\ref{subsection:Seifert} we have
$p'=1$ and $q'=-n$. Hence, we would like to determine the number of tight
contact structures on
\[ M\Bigl(D^2;-\frac{1}{p},\frac{n}{np+1}\Bigr)\;\;\;\text{of boundary slope}
\;\;\;s=\frac{1}{\tb(L)-pq}.\]
(The condition `zero Giroux torsion along the boundary' will be
understood from now on.)
By \cite[Proposition 2.2]{dlz13}, this is the same as the number of tight
structures on
\begin{equation}
\label{eqn:positive}
M\Bigl(D^2;\frac{p-1}{p},\frac{n}{np+1}\Bigr)\;\;\;\text{of boundary slope}
\;\;\;s=\frac{1}{\tb(L)-pq}+1.
\end{equation}
The Seifert invariant $r_1:=(p-1)/p$ lies in the interval
$[\nicefrac{1}{2},1)$, the invariant $r_2:=n/(np+1)$, in $(0,\nicefrac{1}{2})$.
It follows that the only case of the classification in \cite{dlz13} that
applies is their case (1), where the concrete values of the Seifert
invariants $r_1,r_2$ (in the range $(0,1)\cap\Q$) are irrelevant,
but the slope $s$ has to satisfy
\begin{equation*}
\tag{DLZ1} s\in (-\infty,0)\cup[2,\infty).
\end{equation*}
The slope in equation (\ref{eqn:positive}) lies in $[0,2]\cup\{\infty\}$,
so the only case where the classification applies is
\[ \tb(L)=pq+1=np^2+p+1,\]
when $s=2$.

\subsubsection{Negative torus knots}
\label{subsubsection:negative}
We consider negative $(p,q)$-torus knots with $p\geq 2$
and $q=-(np-1)$ for some $n\geq 2$.
We have $p\cdot n+1\cdot q=1$,
so that $p'=1$ and $q'=n$. This requires us to find the tight
contact structures on
\[ M\Bigl(D^2;-\frac{1}{p},\frac{n}{np-1}\Bigr)\;\;\;\text{of boundary slope}
\;\;\;s=\frac{1}{\tb(L)-pq},\]
or on
\begin{equation}
\label{eqn:negative}
M\Bigl(D^2;\frac{p-1}{p},\frac{n}{np-1}\Bigr)\;\;\;\text{of boundary slope}
\;\;\;s=\frac{1}{\tb(L)-pq}+1.
\end{equation}
The value
\[ \tb(L)=pq+1=-np^2+p+1\]
gives us $s=2$, so we are again in the case (DLZ1), with no
restriction on the Seifert invariants.

The Seifert invariant $r_1:=(p-1)/p$ lies in the interval
$[\nicefrac{1}{2},1)$, and for $p\geq 3$ the Seifert invariant
$r_2:=n/(np-1)$, in the
interval $(0,\nicefrac{1}{2})$. This means that no other case of
the classification in \cite{dlz13} applies.

For $p=2$ we have $1/2<r_2<1$ for all~$n$, so we are in case (2)
of~\cite{dlz13}:
\begin{equation*}
\tag{DLZ2}
r_1,r_2\in [\nicefrac{1}{2},1)\;\;\;\text{and}\;\;\;s\in[0,1).
\end{equation*}
With equation (\ref{eqn:negative}) the slope condition $s\in[0,1)$
translates into $\tb(L)<-p(np-1)$.
\section{Exceptional left-handed trefoils}
\label{section:lht}
The discussion in the previous section suggests as potentially
worthwhile the study of exceptional left-handed trefoils,
i.e.\ $(2,-3)$-torus knots, subject to the condition
$\tb(L)=-5$ or $\tb(L)<-6$. The aim of this section is to prove
Theorem~\ref{thm:lht}, which shows that here a
complete classification is indeed possible.

As we shall see in the course of
the proof, explicit realisations of the two exceptional
left-handed trefoils with $\tb<-6$ are given as follows.
The left-handed trefoil $L$ in Figure~\ref{figure:lht-tb-6},
taken from~\cite{geon18}, is strongly exceptional with $\tb(L)=-6$. With
the clockwise orientation it has
$\rot(L)=-7$. Its negative stabilisations are strongly exceptional,
with $(\tb,\rot)=(-6-k,-7-k)$. The second strongly exceptional
realisation is given by reversing the orientation, so that
$(\tb,\rot)=(-6-k,7+k)$.

\begin{figure}[h]
\labellist
\small\hair 2pt
\pinlabel $L$ [l] at 74 20
\pinlabel $+1$ [l] at 74 28
\pinlabel $+1$ [l] at 74 35
\pinlabel $-1$ [l] at 74 55
\pinlabel $-1$ [l] at 74 77
\pinlabel $-1$ [l] at 63 99
\endlabellist
\centering
\includegraphics[scale=1.5]{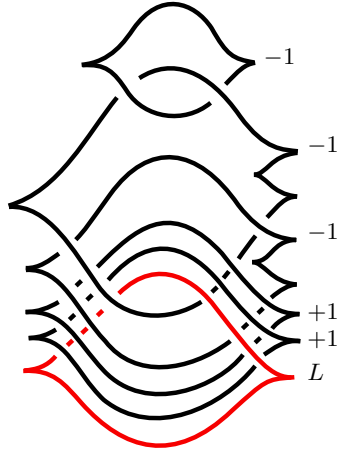}
  \caption{An exceptional left-handed trefoil $L$ with $\tb=-6$.}
  \label{figure:lht-tb-6}
\end{figure}

\subsection{Number of tight structures on the knot complement}
\label{subsection:number}
In the notation of Section~\ref{subsubsection:negative}
we take $p=2$ and $n=2$, so that $q=-(np-1)=-3$.
The knot complement of a Legendrian $(2,-3)$-torus knot $L$
is the Seifert manifold $M(D^2;\nicefrac{1}{2},\nicefrac{2}{3})$, with
boundary slope $s=1+1/(\tb(L)+6)$.
\subsubsection{The case $\tb(L)=-5$}
\label{subsubsection:tb-5}
Here we have slope $s=2$. Following the algorithm of \cite[p.~65]{dlz13},
we write the number $s-[s]$ as $b/a$ with integers $a>b\geq 0$,
so we may take $a=1, b=0$, and then set
\[ a_1=\frac{1}{1-b/a}+1=2,\;\;\; a_2=1.\]
Moreover, by \cite[p.~68]{dlz13}, we have to write down negative
continued fraction expansions of the $-1/r_i$, $i=1,2$. For $r_1=1/2$
this gives us $a_0^1=-2$. For $r_2=2/3$ we have
\[ -\frac{3}{2}=-2-\frac{1}{-2},\]
that is, $a_0^2=a_1^2=-2$.
By formula ($***$) on p.~75 of~\cite{dlz13}, there are exactly
\[ [s]\prod_{i,j}|a_j^i+1|(a_1-1)a_2=2\]
tight contact structures.
\subsubsection{The case $\tb(L)<-6$}
\label{subsubsection:tb<-6}
For $\tb(L)=-6-k$, $k\in\N$, the corresponding slope $s=s_k$ is given by
\[ s=\frac{1}{-k}+1=\frac{k-1}{k}.\]
Now, according to \cite[Theorem~1.1]{dlz13}, the number of tight
contact structures, up to isotopy fixing the boundary, on
$M(D^2;\nicefrac{1}{2},\nicefrac{2}{3})$ equals the number of tight
contact structures, up to isotopy, on the small Seifert manifold
\[ M\Bigl(-1-[s];\frac{1}{2},\frac{2}{3},r_3\Bigr),\]
where the Seifert invariant $r_3$ is determined as follows.

Choose integers $a>b\geq 0$ such that
\[ \frac{b}{a}=s-[s]=\frac{k-1}{k}.\]
This means $a=k$, $b=k-1$. Then
\[ \frac{1}{1-b/a}=k\]
is an integer, so the recipe of \cite[p.~65]{dlz13} again tells us to set
\[ a_1=\frac{1}{1-b/a}+1=k+1,\;\;\; a_2=1,\]
and to define
\[ r_3=\cfrac{1}{a_1-\cfrac{1}{a_2+1}}=\frac{2}{2k+1}.\]

The number of tight contact structures on the Seifert manifold
\[ M\Bigl(-1;\frac{1}{2},\frac{2}{3},\frac{2}{2k+1}\Bigr)\]
has been determined in~\cite[Theorem~1.1]{gls07}.
The formula given there is elementary but quite involved, so we
leave it to the reader to check our calculation, which gives
$k+4$ tight contact structures. Beware that the calculations
in \cite{gls07} presume that the Seifert invariants are ordered in size,
so one needs to work with
\[ (r_1,r_2,r_3)=(\nicefrac{2}{3},\nicefrac{1}{2},\nicefrac{2}{2k+1})\]
for $k\geq 2$, and with
\[ (r_1,r_2,r_3)=(\nicefrac{2}{3},\nicefrac{2}{3},\nicefrac{1}{2})\]
for $k=1$.
\subsection{Realisations in the standard contact structure}
The Legendrian realisations of the left-handed trefoil knot in
$(S^3,\xist)$ have been classified by Etnyre and Honda~\cite{etho01}.
Here is a paraphrase of their Theorems 4.3 and~4.4.

\begin{thm}[Etnyre--Honda]
\label{thm:EH}
The Legendrian realisations $L$ of the left-handed trefoil knot
in $(S^3,\xist)$ are determined, up to Legendrian isotopy,
by their classical invariants in the range
\[ \tb(L)=-6-k,\;\; k\in\N_0,\]
and, correspondingly,
\[ \rot(L)\in\{-(k+1),-(k-1),\ldots,k-1,k+1\}.\]
\end{thm}

This means that for $\tb(L)=-6-k$ there are $k+2$ distinct realisations.
\subsection{The LOSS invariant}
\label{subsection:loss}
In \cite{loss09}, Lisca \emph{et al.}\ introduced an invariant
$\loss$ for oriented, homologically trivial Legendrian knots $L$
in an arbitrary closed contact $3$-manifold $(M,\xi)$, taking
values in certain Heegaard Floer homology groups, with the following
properties:
\begin{itemize}
\item[(i)] If $\loss(L)\neq 0$ and $\xi$ is overtwisted,
then $L$ is exceptional
\cite[Theorem~1.4]{loss09}.
\item[(ii)] The negative stabilisation $S_-L$ of $L$ has
the same LOSS invariant:\\
$\loss(S_-L)=\loss(L)$ \cite[Theorem~1.6]{loss09}.
\end{itemize}

Moreover, in \cite[Theorem~6.8]{loss09} they established that
the Legendrian knot $L$ in Figure~\ref{figure:lht-tb-6}
with the clockwise orientation has $\loss(L)\neq 0$. The rotation number
of $L$ is found with the formula in \cite[Lemma~6.6]{loss09}
to be $\rot(L)=-7$. (See equation~(\ref{eqn:rot}) below for this formula.)

With the properties (i) and (ii) of $\loss$ we conclude that the
negative stabilisations of $L$ give us exceptional Legendrian
left-handed trefoils with $(\tb,\rot)$ taking the values
$(-6-k,-7-k)$, $k\in\N$. Reversing the orientation of these knots
gives exceptional trefoils with $(\tb,\rot)=(-6-k,7+k)$.

\begin{rem}
The vanishing theorem \cite[Corollary~1.2]{stve09}
guarantees that a (homologically trivial)
Legendrian knot $L$ with $\loss(L)\neq 0$ is in fact
\emph{strongly} exceptional.
\end{rem}
\subsection{Proof of Theorem~\ref{thm:lht}}
\label{subsection:proof-lht}
For $\tb(L)=-6-k$, $k\in\N$, we found in Section~\ref{subsubsection:tb<-6}
that there are at most $k+4$ Legendrian realisations with
zero Giroux torsion in the complement. Theorem~\ref{thm:EH}
gives us $k+2$ realisations in~$\xist$, so there can be at most
two strongly exceptional realisations. These are the two described in
Section~\ref{subsection:loss} and detected by the LOSS invariant.

The examples in Section~\ref{subsection:loss} include the case $\tb=-6$
(as shown in Figure~\ref{figure:lht-tb-6}).

\begin{rem}
From our discussion we can conclude that the knot $-L$, i.e.\
the knot in Figure~\ref{figure:lht-tb-6} with the
counter-clockwise orientation has $\loss(-L)=0$. For otherwise
the negative stabilisations of $-L$ would likewise be strongly exceptional.
This would give more strongly exceptional realisations than allowed by
the arithmetic in Section~\ref{subsection:number}. For instance,
for $\tb=-7$ there can be at most five Legendrian realisations
of the left-handed trefoil with complement having zero Giroux torsion.
The three realisations in $\xist$ have $\rot\in\{0,\pm 2\}$.
The exceptional knots $\pm S_-L$ in $\xi_{3/2}$ have $\rot=\pm 8$. The knot
$S_-(-L)=-S_+L$ has $\rot=6$, so it must be distinct from the
other five, and hence loose.
\end{rem}

For $\tb(L)=-5$ there can be at most two strongly
exceptional realisations of the left-handed trefoil by
the calculation in Section~\ref{subsubsection:tb-5}.
Explicit realisations, as we shall explain, are given in
Figure~\ref{figure:lht-tb-5}
(with the two choices of orientation for~$L$).

\begin{figure}[h]
\labellist
\small\hair 2pt
\pinlabel $-1$ [br] at 13 36
\pinlabel $-1$ [b] at 50 38
\pinlabel $L$ [br] at 70 39
\pinlabel $+1$ [tl] at 98 5
\pinlabel $-1$ [bl] at 125 40
\pinlabel $K$ [tr] at 73 5
\endlabellist
\centering
\includegraphics[scale=1.4]{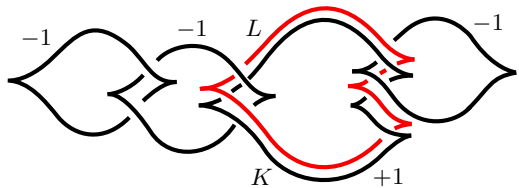}
  \caption{The two exceptional left-handed trefoils $L$ with $\tb=-5$.}
  \label{figure:lht-tb-5}
\end{figure}

For the other values of $\tb(L)$, examples of strongly
exceptional realisations are given in Figure~\ref{figure:lht-tbgeq-5}.
Here $m$ denotes the number of Legendrian unknots with $\tb=-1$
in the vertical chain;
notice that for $m=0$ this figure specialises to
Figure~\ref{figure:lht-tb-5}.

\begin{figure}[h]
\labellist
\small\hair 2pt
\pinlabel $-1$ [br] at 13 144
\pinlabel $-1$ [b] at 50 147
\pinlabel $-1$ [b] at 78 143
\pinlabel $-1$ [bl] at 117 136
\pinlabel $-1$ [tl] at 95 109
\pinlabel $-1$ [l] at 101 70
\pinlabel $-1$ [l] at 101 53
\pinlabel $L$ [l] at 105 39
\pinlabel $+1$ [l] at 105 14
\pinlabel $m$ [r] at 36 91
\pinlabel $K$ [tr] at 59 14
\endlabellist
\centering
\includegraphics[scale=1.5]{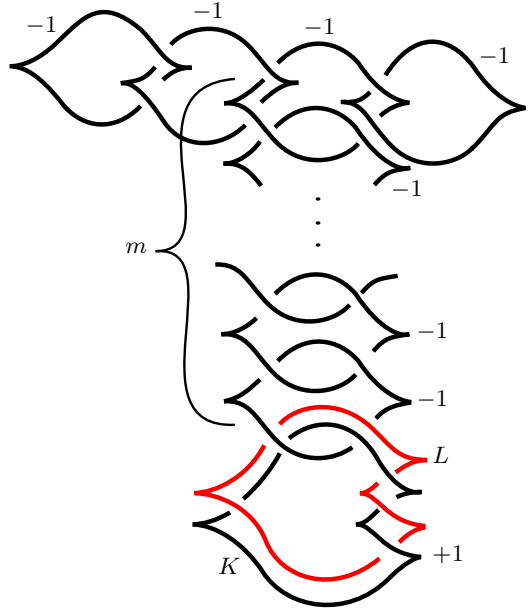}
  \caption{Exceptional left-handed trefoils with $\tb=m-5$.}
  \label{figure:lht-tbgeq-5}
\end{figure}

\begin{lem}
\label{lem:lht}
The knot $L$ in the surgery diagram of Figure~\ref{figure:lht-tbgeq-5} ---
including the case $m=0$, shown separately in Figure~\ref{figure:lht-tb-5} ---
is a left-handed trefoil in the $3$-sphere.
\end{lem}

\begin{proof}
The knot $L$ has linking $-2$ with
the surgery curve $K$ at the bottom of the picture, whose
topological surgery framing is~$-1$. By blowing up twice, so that we
place two $(+1)$-framed meridians around~$K$, 
we can turn $L$ into a parallel curve of $K$ (with zero linking),
and the surgery framing of $K$ has changed to~$+1$. A handle slide of $L$
over $K$ will then turn $L$ into a meridian of~$K$,
unlinked from the next surgery curve and the two $(+1)$-framed
meridians. Blowing down the
two meridians will return the old surgery framing $-1$ of $K$; the
resulting situation is shown in Figure~\ref{figure:lht-Kirby}.

\begin{figure}[h]
\labellist
\small\hair 2pt
\pinlabel $L$ [br] at 2 20
\pinlabel $-1$ [b] at 18 21
\pinlabel $-2$ [b] at 37 23
\pinlabel $-2$ [b] at 51 23
\pinlabel $-2$ [b] at 91 23
\pinlabel $-2$ [tl] at 96 2
\pinlabel $-2$ [b] at 103 25
\pinlabel $-2$ [bl] at 118 24
\pinlabel $K$ [t] at 20 11
\endlabellist
\centering
\includegraphics[scale=1.9]{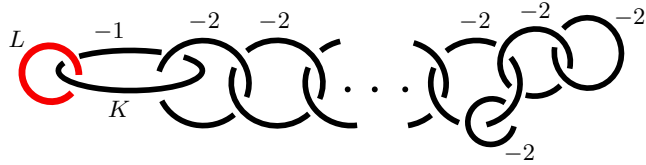}
  \caption{The topological picture for Figures
           \ref{figure:lht-tb-5} and~\ref{figure:lht-tbgeq-5}.}
  \label{figure:lht-Kirby}
\end{figure}

For $m=0$, this is the same picture as in~\cite[Figure~3]{geon18}.
For $m\geq 1$, blowing down $K$ will increase the surgery framing of the
next unknot in the chain to~$-1$. We continue in this fashion
with $m-1$ further blow-downs until we reach the same
picture as in the case $m=0$.
The further Kirby moves that turn this into the picture
of a left-handed trefoil in $S^3$
are shown in Figures 3 and~4 of~\cite{geon18}.
\end{proof}

\begin{lem}
\label{lem:lht-rot}
The Legendrian knot $L$ in Figure~\ref{figure:lht-tbgeq-5} has
$\tb(L)=m-5$ and, depending on a choice of orientation, $\rot(L)=\pm(m-6)$.
\end{lem}

\begin{proof}
Although the final topological picture does not change
by adding Legendrian unknots with $\tb=-1$ and contact surgery
coefficient $-1$ to the vertical chain in Figure~\ref{figure:lht-tbgeq-5},
each of these unknots, when it is blown down, increases the
framing of $L$ by one.
It therefore suffices to show that $\tb(L)=-5$ for the
Legendrian knot $L$ shown in Figure~\ref{figure:lht-tb-5}, which can be
done with the formula from \cite[Lemma~6.6]{loss09},
cf.~\cite[Lemma~3.1]{geon15} and \cite{duke16}.

Let $M$ be the linking matrix of the surgery diagram in
Figure~\ref{figure:lht-tb-5}. When we order the surgery knots
as $L_1,\ldots,L_4$ from
left to right and orient them clockwise, this linking matrix is
\[ M=\left(\begin{array}{cccc}
-2 &  1 &  0 &  0\\
 1 & -2 &  1 &  0\\
 0 &  1 & -1 &  1\\
 0 &  0 &  1 & -2
\end{array}\right).\]
The extended linking matrix $M_0$ is defined by including $L$ as the
first knot in the diagram, with self-linking number set to zero:
\[ M_0=\left(\begin{array}{ccccc}
 0 &  0 &  1 & -2 &  1\\
 0 & -2 &  1 &  0 &  0\\
 1 &  1 & -2 &  1 &  0\\
-2 &  0 &  1 & -1 &  1\\
 1 & 0 &  0 &  1 & -2
\end{array}\right).\]
Write $\tb_0$ for the Thurston--Bennequin invariant of $L$ as
a knot in the unsurgered copy of~$S^3$. Then the formula from
\cite{loss09} for the Thurston--Bennequin invariant of the
left-handed trefoil knot $L$ in the surgered copy of $S^3$ gives
\begin{eqnarray}
\label{eqn:tb}
\tb(L) & = & \tb_0+\frac{\det M_0}{\det M}\\
       & = & -2+\frac{3}{-1}\;=\;-5.\nonumber
\end{eqnarray}

For the rotation number $\rot(L)$ of $L$ in the surgered $S^3$ we also
have a formula from~\cite{loss09}; cf.~\cite{duke16}.
Write $\rot_0$ for the rotation number of $L$ in the unsurgered copy
of~$S^3$. Then
\begin{eqnarray}
\label{eqn:rot}
\rot(L) & = & \rot_0-\left\langle\left(\begin{array}{c}
              \rot(L_1)\\ \vdots\\ \rot(L_4)\end{array}\right),
              M^{-1}\left(\begin{array}{c}
              \lk(L,L_1)\\ \vdots\\ \lk(L,L_4)\end{array}\right)
              \right\rangle\\
        & = & 1-\left\langle\left(\begin{array}{c}
              0\\0\\1\\0\end{array}\right),
              M^{-1}\left(\begin{array}{c}
              0\\1\\-2\\1\end{array}\right)\right\rangle\nonumber\\
        & = & \langle(0,0,1,0)^{\ttt},(-2,-4,-5,-3)^{\ttt}\rangle
              \;\; = \;\; 1+5\;\; = \;\; 6.\nonumber
\end{eqnarray}
The computation of the rotation number for the general case shown in
Figure~\ref{figure:lht-tbgeq-5} is analogous, see
Section~\ref{subsection:compute-lht}.
\end{proof}

Notice that for $m=6$, when $\tb(L)=1$,
we have $\rot(L)=0$, so we cannot distinguish
the two orientations of~$L$.

The lacuna in the
classification of tight contact structures on the trefoil
complement prevents us from concluding that the examples
for $m\geq 1$ constitute
a comprehensive list in the range~$\tb>-5$.

The next lemma says that $L$ lives in an overtwisted contact
structure, and as the push-off of the single
contact $(+1)$-surgery curve, it must then be
strongly exceptional by Proposition~\ref{prop:strongly}.
This completes the proof of Theorem~\ref{thm:lht}.

\begin{lem}
\label{lem:lht-d3}
The Legendrian knot $L$ in Figure~\ref{figure:lht-tbgeq-5}
lives in the overtwisted contact structure $\xi_{3/2}$.
\end{lem}

\begin{proof}
We carry out the computation for the case $m=0$ shown in
Figure~\ref{figure:lht-tb-5}; for the general case see
Section~\ref{subsection:compute-lht}.
Read as a Kirby diagram, there are four $2$-handles in this figure,
so it describes a $2$-handlebody $X$ of Euler characteristic $\chi(X)=5$.
The signature of this $4$-manifold $X$ is $\sigma(X)=-2$. This can be computed
as the signature of the linking matrix~$M$; better, it can be
determined by keeping track of the (positive or negative) blow-ups
during the Kirby moves. By \cite[Corollary~3.6]{dgs04}, the
$d_3$-invariant of the contact structure $\xi$ described by the
surgery diagram is then given by
\begin{equation}
\label{eqn:d_3}
d_3(\xi)=\frac{1}{4}\bigl(c^2-3\sigma(X)-2\chi(X)\bigr)+q,
\end{equation}
where $q$ denotes the number of contact $(+1)$-surgeries (here $q=1$),
and $c\in H^2(X;\Z)$ is the cohomology class which evaluates
as $\rot(L_i)$ on the homology generator of $X$ defined
by a Seifert surface of the surgery knot $L_i$ glued with the core disc in
the corresponding $2$-handle. Write $\vrot$
for the vector of rotation numbers of the~$L_i$. Then
the square $c^2\in\Z$ is found as $c^2=\bfx^{\ttt}M\bfx$, where $\bfx$
is a solution of the linear equation $M\bfx=\vrot$.

In our situation, we have $\vrot=(0,0,1,0)^{\ttt}$
and $\bfx=(2,4,6,3)^{\ttt}$. This gives
$c^2=\bfx^{\ttt}\vrot=6$ and
\[ d_3(\xi)=\frac{1}{4}(6+6-10)+1=3/2,\]
that is, $\xi=\xi_{3/2}$.
\end{proof}
\section{Exceptional right-handed trefoils}
\label{section:rht}
We now consider Legendrian realisations $L$ of
the right-handed trefoil, i.e.\ the $(2,3)$-torus
knot. We show that any value of $\tb$ can be realised by
an exceptional right-handed trefoil, and we give a complete
classification for one value of~$\tb$.

By Section~\ref{subsubsection:positive}, for $\tb(L)=7$ the knot complement
is the Seifert manifold
\[ M\Bigl(D^2;\frac{1}{2},\frac{1}{3}\Bigr)\;\;\;\text{of boundary slope}
\;\;\; s=2.\]
Computing as in Section~\ref{subsubsection:tb-5}, where now
$-1/r_2=-3$ gives us $a_0^2=-3$, one finds that there are
at most four strongly exceptional realisations.

\begin{prop}
There are exactly four strongly exceptional realisations $L$ of the
right-handed trefoil knot with $\tb(L)=7$. They are shown in
Figure~\ref{figure:rht-tb7} (with either orientation of~$L$).
\end{prop}

\begin{figure}[h]
\labellist
\small\hair 2pt
\pinlabel $-1$ [br] at 18 67
\pinlabel $-1$ [br] at 129 67
\pinlabel $-1$ [bl] at 75 41
\pinlabel $-1$ [bl] at 180 41
\pinlabel $+1$ [tl] at 48 6
\pinlabel $+1$ [tl] at 153 6
\pinlabel $L$ [br] at 13 33
\pinlabel $L$ [br] at 118 33
\endlabellist
\centering
\includegraphics[scale=1.4]{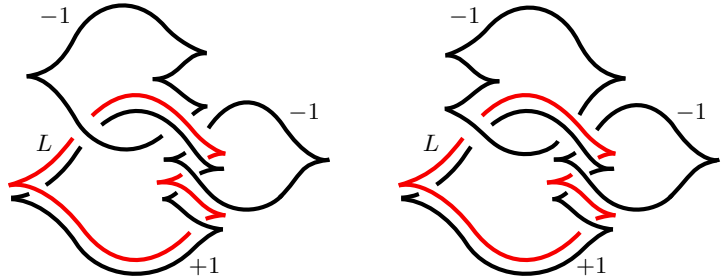}
  \caption{The four right-handed exceptional trefoils with $\tb=7$.}
  \label{figure:rht-tb7}
\end{figure}

\begin{proof}
Topologically, the surgery diagram consists of a chain of
three unknots with surgery coefficients $-3, -1, -2$. After thrice
blowing down a $(-1)$-curve we obtain the $3$-sphere. This also
shows that the handlebody $X$ described by the diagram has
signature $\sigma(X)=-3$.

The Kirby moves for showing that $L$ is topologically a right-handed
trefoil in the surgered $S^3$ are analogous to those
in the proof of Lemma~\ref{lem:lht}.

Straightforward computations yield the following classical invariants
and $d_3$-invariant.
The Legendrian realisation $L$ on the left in Figure~\ref{figure:rht-tb7}
has $(\tb,\rot)=(7,\pm 4)$; the one on the right, $(\tb,\rot)=(7,\pm 8)$.
The overtwisted contact structure of the surgered $S^3$ on the left is
$\xi_{1/2}$; on the right, $\xi_{-3/2}$.
\end{proof}

Although we have no tools yet to classify exceptional right-handed
trefoils $L$ with $\tb(L)\neq 7$, we can say something about their existence
for all values of~$\tb$.

An example with $(\tb,\rot)=(6,\pm 7)$ in $(S^3,\xi_{-3/2})$
has been described in~\cite[Figure~9]{loss09}. With one
of its orientations, this has non-zero LOSS invariant, so by taking
negative stabilisations (and then either orientation) we obtain
a pair of exceptional right-handed trefoils for all values of
$\tb\leq 6$. By putting the two zigzags of the stabilised knot
in \cite[Figure~9]{loss09} to the other side, or one zigzag on
either side, one obtains exceptional right-handed trefoils
in $(S^3,\xi_{1/2})$ with $(\tb,\rot)=(6,\pm 1)$ and $(6,\pm 3)$,
respectively. We do not know, however, whether
these new examples have non-zero LOSS invariant.

The first examples with $\tb>7$ are shown in Figure~\ref{figure:rht-tbgeq7}.
For $m=0$ these examples reduce to those in Figure~\ref{figure:rht-tb7},
and the same argument as for the left-handed trefoils shows why
the Thurston--Bennequin invariant takes the value $m+7$.
The computation of the other invariants is completely standard, and we only
give the results in Table~\ref{table:rht}. We summarise our findings in the
following proposition.

\begin{prop}
Any integer can be realised as the Thurston--Bennequin invariant of
a strongly exceptional Legendrian realisation of the
right-handed trefoil. All the known examples for $\tb\leq 5$ live
in $\xi_{-3/2}$. For each $\tb\geq 6$, there are examples
in $\xi_{-3/2}$ and~$\xi_{1/2}$.
\qed
\end{prop}

\begin{figure}[h]
\labellist
\small\hair 2pt
\pinlabel $(\rma)$ at 0 167
\pinlabel $L$ [l] at 64 38
\pinlabel $+1$ [l] at 64 14
\pinlabel $-1$ [l] at 61 53
\pinlabel $-1$ [l] at 61 70
\pinlabel $-1$ [tl] at 53 109
\pinlabel $-1$ [bl] at 77 135
\pinlabel $-1$ [br] at 36 144
\pinlabel $-1$ [bl] at 47 163
\pinlabel $m$ [r] at 0 91
\pinlabel $(\rmb)$ at 131 167
\pinlabel $L$ [l] at 197 38
\pinlabel $+1$ [l] at 197 14
\pinlabel $-1$ [l] at 194 53
\pinlabel $-1$ [l] at 194 70
\pinlabel $-1$ [tl] at 187 109
\pinlabel $-1$ [bl] at 210 135
\pinlabel $-1$ [bl] at 171 144
\pinlabel $-1$ [br] at 155 163
\pinlabel $m$ [r] at 131 91
\endlabellist
\centering
\includegraphics[scale=1.4]{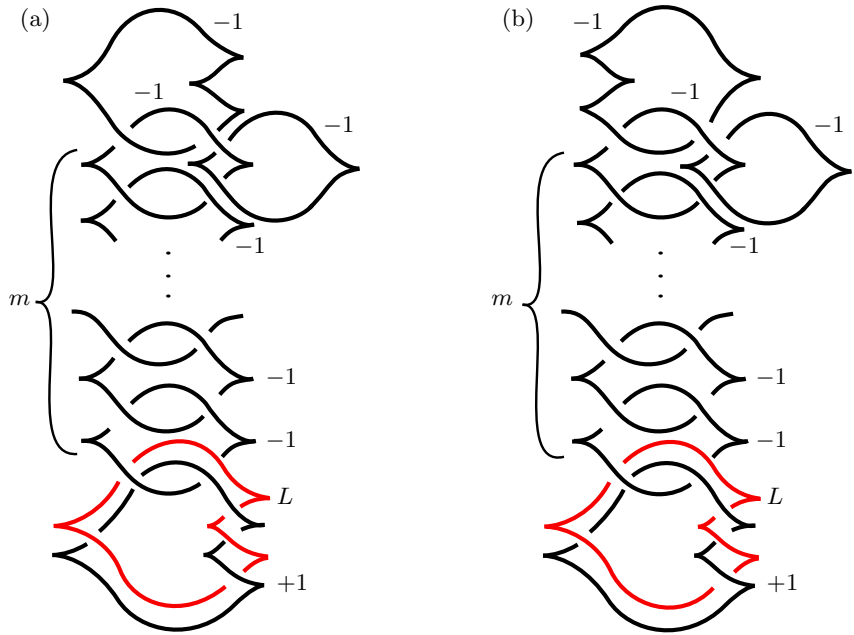}
  \caption{Right-handed exceptional trefoils with $\tb>7$.}
  \label{figure:rht-tbgeq7}
\end{figure}

\begin{table}
{\renewcommand{\arraystretch}{1.6}
\begin{tabular}{|c|c|c|c|c|}  \hline
Figure                     & $m$  & $\tb$ & $\rot$     & $d_3$ \\ \hline
\ref{figure:rht-tbgeq7}(a) & odd  & $m+7$ & $\pm(m+1)$ & $-3/2$ \\ \hline
\ref{figure:rht-tbgeq7}(a) & even & $m+7$ & $\pm(m-3)$ & $1/2$  \\ \hline
\ref{figure:rht-tbgeq7}(b) & odd  & $m+7$ & $\pm(m-3)$ & $1/2$  \\ \hline
\ref{figure:rht-tbgeq7}(b) & even & $m+7$ & $\pm(m+1)$ & $-3/2$  \\ \hline
\end{tabular}
}
\vspace{1.5mm}
\caption{Invariants of the right-handed trefoils in
Figure~\ref{figure:rht-tbgeq7}.}
\label{table:rht}
\end{table}

\section{General torus knots}
\label{section:general}
In this section we prove two classification results for
strongly exceptional realisations of general torus knots.

\begin{prop}
\label{prop:positive}
For $p\geq 2$ and $n\geq 1$,
there are exactly $2p$ strongly exceptional Legendrian realisations $L$
of the $(p,np+1)$-torus knot with $\tb(L)=np^2+p+1$. They are shown
in Figure~\ref{figure:pnp-plus-1} (with either orientation of~$L$),
where $k,l\in\N_0$ with $k+l=p-1$.
The rotation number of these knots is
\[ \rot(L)=\pm\bigl(np^2+p-np(l-k)\bigr).\]
The ambient overtwisted contact structure $\xi$ is determined by
\[ d_3(\xi)=\frac{n}{4}\bigl(1-(p-l+k)^2\bigr)+\frac{1}{2}.\]
\end{prop}

\begin{figure}[h]
\labellist
\small\hair 2pt
\pinlabel $-1$ [l] at 59 16
\pinlabel $-1$ [l] at 59 23
\pinlabel $+1$ [l] at 59 32
\pinlabel $L$ [l] at 59 56 
\pinlabel $-1$ [bl] at 42 111 
\pinlabel $-1$ [tr] at 40 106
\pinlabel $k$ [r] at 3 92
\pinlabel $l$ [l] at 59 92
\pinlabel $n$ [r] at 30 117
\pinlabel $p-1\biggl\{\biggr.$ [r] at 11 19
\endlabellist
\centering
\includegraphics[scale=1.8]{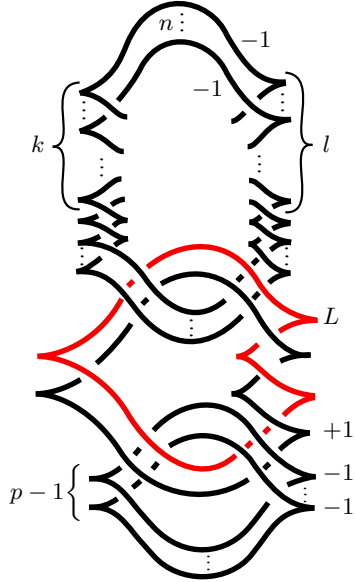}
  \caption{The $2p$ exceptional $(p,np+1)$-torus knots with
           $\tb=np^2+p+1$.}
  \label{figure:pnp-plus-1}
\end{figure}

\begin{proof}
In the notation of Section~\ref{subsubsection:positive}
we have Seifert invariants $r_1=(p-1)/p$ and $r_2=n/(np+1)$,
and boundary slope $s=2$. With the continued fraction expansions
\[ -\frac{1}{r_1}=[\underbrace{-2,\ldots,-2}_{p-1}]\;\;\;\text{and}\;\;\;
-\frac{1}{r_2}=[-(p+1),\underbrace{-2,\ldots,-2}_{n-1}] \]
the formula in Section~\ref{subsubsection:tb-5} tells us
there are $2p$ contact structures of zero Giroux torsion on the knot
complement. So we need only verify that Figure~\ref{figure:pnp-plus-1} does
indeed show $2p$ distinct strongly exceptional realisations of the
$(p,np+1)$-torus knot.

Topologically, we transform the diagram by the same procedure
as that used at the beginning of Lemma~\ref{lem:lht}.
This produces a Kirby diagram as in \cite[Figure~18]{onar10},
and the further Kirby moves pictured there demonstrate that
$L$ is a $(p,np+1)$-torus knot in~$S^3$.

The classical invariants of $L$ and the $d_3$-invariant of
the ambient contact structure are computed in
Section~\ref{subsection:compute+}.
The $2p$ realisations are distinguished by the rotation number.
\end{proof}

\begin{rem}
We take the opportunity to point out a minor correction to
Figure~18 of~\cite{onar10}. The $n-1$ meridians with surgery
framing $-2$ ought to link one another $-1$ times.
\end{rem}

\begin{prop}
\label{prop:p-np-1}
For $p\geq 2$ and $n\geq 2$, there are exactly $2(p-1)(n-1)$ strongly
exceptional realisations $L$ of the $(p,-(np-1))$-torus knot
with $\tb(L)=-np^2+p+1$. They are shown in Figure~\ref{figure:p-np-1}
(with either orientation of~$L$), where $k,l,u,v\in\N_0$
with $k+l=p-2$ and $u+v=n-2$. The rotation number of these knots is
\[ \rot(L)=\pm\bigl(np^2-p-np(l-k)+p(v-u)\bigr).\]
The ambient overtwisted contact structure $\xi$ is determined by
\[ d_3(\xi)=\frac{1}{4}\bigl(n(p-l+k)^2+2(p-l+k)(v-u)\bigr)-\frac{1}{2}.\]
\end{prop}

\begin{figure}[h]
\labellist
\small\hair 2pt
\pinlabel $L$ [l] at 59 56
\pinlabel $-1$ [l] at 59 16
\pinlabel $-1$ [l] at 59 23
\pinlabel $+1$ [l] at 59 32
\pinlabel $p-1\biggl\{\biggr.$ [r] at 10 19
\pinlabel $k$ [r] at 0 92
\pinlabel $l$ [l] at 63 91
\pinlabel $-1$ [bl] at 44 110
\pinlabel $u$ [r] at 0 144
\pinlabel $v$ [l] at 58 143
\pinlabel $-1$ [bl] at 39 164
\endlabellist
\centering
\includegraphics[scale=1.4]{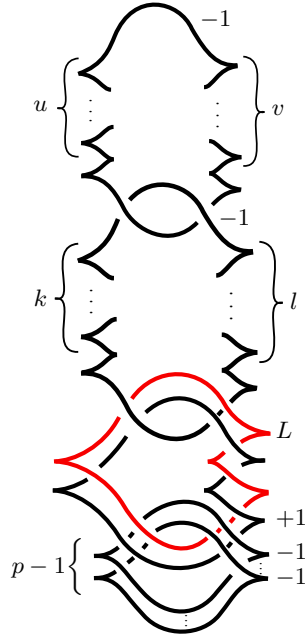}
  \caption{The $2(n-1)(p-1)$ exceptional $(p,-(np-1))$-torus knots
           with $\tb=-np^2+p+1$.}
  \label{figure:p-np-1}
\end{figure}

\begin{proof}
The Seifert invariants are $r_1=p/(p-1)$ and $r_2=n/(np-1)$,
and the boundary slope is $s=2$.
We have the continued fraction expansions
\[ -\frac{1}{r_1}=[\underbrace{-2,\ldots,-2}_{p-1}]\;\;\;\text{and}\;\;\;
-\frac{1}{r_2}=[-p,-n].\]
The formula in Section~\ref{subsubsection:tb-5} gives us $2(p-1)(n-1)$
distinct contact structures of zero Giroux torsion on the
knot complement. Thus,
we need to verify that Figure~\ref{figure:p-np-1} shows
$2(p-1)(n-1)$ distinct Legendrian realisations of the
$(p,-(np-1))$-torus knot.

As before, topologically we transform $L$ into a meridian of
the parallel knot. We then have the same diagram as
in~\cite[Figure~7]{geon18}, where it is shown that $L$
is a $(p,-(np-1))$-torus knot.

For the remaining calculations see Section~\ref{subsection:compute-}.
\end{proof}

\begin{rem}
(1) By \cite[Theorem~4.1]{etho01}, the maximal Thurston--Bennequin invariant
for realisations of the $(p,q)$-torus knot in $(S^3,\xist)$ is
\[ \otb_{(p,q)}=pq-p-q\;\;\text{for $p,q>0$}\]
and
\[ \otb_{(p,q)}=pq\;\;\text{for $p>0$, $q<0$}.\]
In particular, we have
\[ \otb_{(p,np+1)}=np^2-np-1<np^2+p+1\]
and
\[ \otb_{(p,-(np-1))}=-np^2+p<-np^2+p+1,\]
from which we can deduce directly that the contact structures
on $S^3$ described by the surgery diagrams in Figures~\ref{figure:pnp-plus-1}
and~\ref{figure:p-np-1} must be overtwisted.

(2) Alternative descriptions of the exceptional $(p,-(np-1))$-torus
knots can be found in~\cite[Figure~6]{geon18}. In that paper,
we showed that, ignoring the orientation of~$L$,
the $(p-1)(n-1)$ realisations can be distinguished by
the result of performing contact $(-1)$-surgery on them: one
obtains the lens space  $L(np^2-p+1,p^2)$ with $(p-1)(n-1)$
pairwise homotopically distinct contact structures, detected
by the Euler class. To distinguish the different orientations
of~$L$, one needs the rotation number.
\end{rem}
\section{Some computations}
\label{section:computations}
In this section we collect a few more details about the
calculations in the general cases of the previous sections.
\subsection{Left-handed trefoils}
\label{subsection:compute-lht}
For the computation of the rotation number in Lem\-ma~\ref{lem:lht-rot}
for arbitrary $m$ we order the knots in the surgery diagram
(Figure~\ref{figure:lht-tbgeq-5})
from the bottom to the top of the vertical chain, followed by the three
knots at the top, starting at the right. Then the linking matrix $M$
(with all knots oriented clockwise) becomes

\[ M=\left(\begin{array}{ccccccccc}
-1     & \makebox[0pt][l]{$\smash{\overbrace{
        \phantom{\begin{matrix}-2&-2&\ldots&-2&-2-\end{matrix}}}^{\text{$m$}}}$}
             -1 &  0     & \ldots &  0 & \ldots & \ldots & \ldots & 0\\
-1     & -2     & -1     & \ddots &  0 & \ldots & \ldots & \ldots & \vdots\\
0      & -1     & -2     & \ddots &  0 & \ldots & \ldots & \ldots & \vdots\\
\vdots &  0     & -1     & \ddots & -1 &  0     & \ldots & \ldots & \vdots\\
\vdots & \vdots &  0     & \ddots & -2 & -1     & 0      & 0      & 0\\
\vdots & \vdots & \vdots & \ddots & -1 & -2     & 1      & 1      & 0\\
\vdots & \ldots & \ldots & \ldots &  0 &  1     & -2     & 0      & 0\\
\vdots & \ldots & \ldots & \ldots &  0 &  1     & 0      & -2     & 1\\
0      & \ldots & \ldots & \ldots &  0 &  0     & 0      & 1      & -2
\end{array}\right).\]
The vectors of rotation and linking numbers, with $L$ also
oriented clockwise, are
\[ \vrot=(1,0,\ldots,0)^{\ttt}\;\;\;
\text{and}\;\;\;
\vlk=(-2,-1,0,\ldots,0)^{\ttt},\]
respectively. By the formula (\ref{eqn:rot}) for $\rot(L)$ in the proof
of Lemma~\ref{lem:lht-rot} (or rather the obvious generalisation
of this formula to any number of surgery knots), we need only compute
the first entry of $M^{-1}\vlk$, for which it suffices
to know the first row of $M^{-1}$. By the symmetry of $M$ (and
hence~$M^{-1}$), this is the
same as the first column of $M^{-1}$, that is, $M^{-1}(1,0,\ldots,0)^{\ttt}$.

It is easy to verify that $\bfx:=M^{-1}(1,0,\ldots,0)^{\ttt}$ equals
\[ \bfx=\bigl(6-m,-(7-m),8-m,\ldots,(-1)^m\cdot 6,
(-1)^m\cdot 3, (-1)^m\cdot 4,(-1)^m\cdot 2\bigr)^{\ttt}.\]
Then
\[ \rot(L)=1-(-2,-1,0,\ldots,0)\bfx=6-m.\]

Next we compute the $d_3$-invariant for the surgered $S^3$ shown in
Figure~\ref{figure:lht-tbgeq-5}, which we claimed in Lemma~\ref{lem:lht-d3}
to equal $3/2$ for all $m\in\N$. For $m=0$ we had $\chi(X)=5$
and $\sigma(X)=-2$. Each additional $2$-handle adds $1$ to the Euler
characteristic, so the corresponding handlebody $X_m$ has
$\chi(X_m)=m+5$. The additional surgery curves correspond to negative
blow-ups, which gives $\sigma(X_m)=-2-m$. The square of the first
Chern class is computed as $\bfx^{\ttt}M\bfx$, where $\bfx$ is a solution
of $M\bfx=\vrot$. So this is indeed the vector $\bfx$
we have found above, which yields $c^2=6-m$. Putting this information
into the formula (\ref{eqn:d_3}) for $d_3$ we find the the claimed value.

\subsection{Positive torus knots}
\label{subsection:compute+}
In Figure~\ref{figure:pnp-plus-1} we take the `shark' parallel to $L$
as the first knot, followed by the $n$ parallel unknots at the top
and the $p-1$ parallel unknots at the bottom. Then the
linking matrix $M$ (with all knots oriented clockwise) takes
the form

\[ M=\left(\begin{array}{ccccccc}
-1  &
\makebox[0pt][l]{$\smash{\overbrace{\phantom{
\begin{matrix}-p-1&\ldots&-p-\end{matrix}}}^{\text{$n$}}}$}
     -1 &  \ldots & -1 &
\makebox[0pt][l]{$\smash{\overbrace{\phantom{
\begin{matrix}-2&\ldots&-2..\end{matrix}}}^{\text{$p-1$}}}$}
  -1 & \ldots & -1\\
-1     & -p-1   & \ldots & -p     & 0      & \ldots & 0\\
\vdots & \vdots & \ddots & \vdots & \vdots & \vdots & \vdots \\
-1     &  -p    & \ldots & -p-1   & 0      & \ldots & 0\\
-1     &  0     & \ldots & 0      & -2     & \ldots & -1\\
\vdots & \vdots & \vdots & \vdots & \vdots & \ddots & \vdots\\
-1     & 0      & \ldots & 0      & -1     & \ldots & -2\\
\end{array}\right).\]
Here all off-diagonal elements in the two quadratic subblocks
of size $n$ and $p-1$ are $-p$
and $-1$, respectively.

By elementary row and column reduction one checks that
$\det M=(-1)^{n+p}$. Similarly, for the extended matrix
$M_0$ (as defined in the proof of Lemma~\ref{lem:lht-rot}), one
finds $\det M_0=(-1)^{n+p}(np^2+p+3)$. With (\ref{eqn:tb}) this yields
\[ \tb(L)=\tb_0+\frac{\det M_0}{\det M}=np^2+p+1.\]

In order to compute the rotation number with formula~(\ref{eqn:rot}),
we first need to determine
$M^{-1}\vlk$, where $\vlk=(-2,-1,\ldots,-1)^{\ttt}$ is the vector of
linking numbers of $L$ with the $n+p$ surgery knots.
This computation can be simplified
by summing over the two boxes of size $n$ and $p-1$ in~$M$. Thus, we
define the `deflated' matrix
\[ M':=\left(\begin{array}{ccc}
-1 & -n    & -p+1\\
-1 & -np-1 & 0\\
-1 & 0     & -p
\end{array}\right)\]
and solve the equation
$M'\bfy=(-2,-1,-1)^{\ttt}$. Notice that $(-2,-1,-1)^{\ttt}$
is the deflated vector of linking numbers. This gives
\[ \bfy=(np^2+p+1,-p,-np-1)^{\ttt}.\]
Write $\vrot=(1,l-k,0)^{\ttt}$ for the deflated vector of rotation
numbers. With formula~(\ref{eqn:rot}) we obtain
\[ \rot(L)=\rot_0-\vrot^{\ttt}
\left(\begin{array}{ccc}
1 & 0 & 0\\
0 & n & 0\\
0 & 0 & p-1
\end{array}\right)\bfy=np(l-k)-np^2-p,\]
or the negative of that for the counter-clockwise orientation of $L$.
As $k$ and $l$ range over $\N_0$ subject to the condition
$k+l=p-1$, some simple arithmetic shows that this gives $2p$ distinct
values for the rotation number.

\begin{rem}
\label{rem:DK}
The contact $(-1)$-surgeries along the $n$ or $p-1$ parallel knots
in Figure~\ref{figure:pnp-plus-1} (which are Legendrian push-offs of one
another) are equivalent, by the algorithm of~\cite{dgs04}, to a contact
$(-1/n)$- or $(-1/(p-1))$-surgery along a single copy of the respective knot.
For diagrams involving contact $(1/k)$-surgeries, $k\in\Z$, one can directly
apply the formula in \cite[Theorem~2.2]{duke16} to compute $\rot(L)$.
For the diagram at hand, that formula is identical to the one
we obtained above by `deflation'.
\end{rem}

For the computation of the $d_3$-invariant we first observe that
the handlebody $X$ described by the surgery diagram in
Figure~\ref{figure:pnp-plus-1} has Euler characteristic $\chi(X)=1+n+p$,
since the number of $2$-handles is $n+p$, and signature $\sigma(X)=-n-p$,
as can be seen from the $n+p$ negative blow-downs
during the Kirby moves in~\cite[Figure~18]{onar10}.
The solution $\bfx$ of the equation $M'\bfx=\vrot$ is
\[ \bfx=\left(\begin{array}{c}
-np^2+np(l-k)-p\\
p-(l-k)\\
np-n(l-k)+1
\end{array}\right).\]
This gives
\[ c^2=\bfx^{\ttt}\left(\begin{array}{ccc}
1 & 0 & 0\\
0 & n & 0\\
0 & 0 & p-1
\end{array}\right)
\vrot=-n(p-l+k)^2-p.\]
Formula (\ref{eqn:d_3}) for the $d_3$-invariant (in the proof of
Lemma~\ref{lem:lht-d3}) then yields the value of~$d_3$
as claimed in Proposition~\ref{prop:positive}. Notice that the
$d_3$-invariant never attains the value $-1/2=d_3(\xist)$,
so the contact structures on $S^3$ given by these surgery diagrams
are overtwisted.

\begin{rem}
(1) We may check our calculation of the $d_3$-invariant against the formula
in \cite[Theorem~5.1]{duke16} by translating the surgeries
in Figure~\ref{figure:pnp-plus-1} into rational surgeries as explained
in Remark~\ref{rem:DK}. Our deflated matrix $M'$ is precisely
the matrix $Q$ on page 525 of \cite{duke16}, obtained from the
rational surgery coefficients. The characteristic polynomial
of this matrix has three negative real roots, so
in the notation of \cite{duke16} we have $\sigma(M')=-3$. The
rational contact surgery coefficients are $+1$, $-1/n$, and $-1/(p-1)$.
Plugging this into the formula of \cite[Theorem~5.1]{duke16}, we obtain
\begin{eqnarray*}
d_3(\xi) & = & \frac{1}{4}\bigl( c^2+(3-1)-(3-n)-(3-(p-1))\bigr)
               -\frac{3}{4}\sigma(M')-\frac{1}{2}\\[.5mm]
         & = & \frac{n}{4}\bigl(1-(p-l+k)^2\bigr)+\frac{1}{2},
\end{eqnarray*}
as in Proposition~\ref{prop:positive}.

(2) A formula for computing the Thurston--Bennequin invariant of
Legendrian knots represented in surgery diagrams involving rational
contact surgeries was found in~\cite{kege16}. This formula tells
us that, with $\vlk=(-2,-1,-1)^{\ttt}$ denoting the vector of
linking numbers in the rational diagram (i.e.\ the deflated vector),
\[ \tb(L)=\tb_0-\bfy^{\ttt}\left(\begin{array}{ccc}
1 & 0 & 0\\
0 & n & 0\\
0 & 0 & p-1
\end{array}\right)\vlk=np^2+p+1,\]
confirming our earlier computation.
\end{rem}
\subsection{Negative torus knots}
\label{subsection:compute-}
In Figure~\ref{figure:p-np-1} we take the knot parallel to $L$
as the first knot, followed by the $p-1$ unknots at the bottom,
and finally the two knots above~$L$. With all knots oriented clockwise,
the linking matrix is\\[.8mm]
\[ M=\left(\begin{array}{cccccc}
-1  &
\makebox[0pt][l]{$\smash{\overbrace{\phantom{
\begin{matrix}-1&\ldots&-1..\end{matrix}}}^{\text{$p-1$}}}$}
         -1     & \ldots & -1     & -1     & 0\\
-1     & -2     & \ldots & -1     & 0      & 0\\
\vdots & \vdots & \ddots & \vdots & \vdots & \vdots\\
-1     & -1     & \ldots & -2     & 0      & 0\\
-1     & 0      & \ldots & 0      & -p     & -1\\
0      & 0      & \ldots & 0      & -1     &-n
\end{array}\right).\]
This has $\det M=(-1)^{p+1}$, and the determinant of the extended matrix
$M_0$ is $\det M_0=(-1)^{p-1}(-np^2+p+3)$. By (\ref{eqn:tb})
the Thurston--Bennequin invariant of $L$ in the surgered $S^3$ is
\[ \tb(L)=\tb_0+\frac{\det M_0}{\det M}=-np^2+p+1.\]
We now sum over the quadratic subblock of size $p-1$ to obtain the
deflated matrix
\[ M':=\left(\begin{array}{cccc}
-1 & -p+1 & -1 & 0\\
-1 & -p   & 0  & 0\\
-1 & 0    & -p & -1\\
0  & 0    & -1 & -n
\end{array}\right).\]
The deflated vectors of rotation numbers and linking numbers are
\[ \vrot=(1,0,l-k,v-u)^{\ttt}\;\;\;\text{and}\;\;\;
\vlk=(-2,-1,-1,0)^{\ttt},\]
respectively. The solution $\bfy$ of $M'\bfy=\vlk$ is
\[ \bfy= \left(\begin{array}{c}
-np^2+p+1\\
np-1\\
np\\
-p\end{array}\right),\]
so we find
\[ \rot(L)=\rot_0-\vrot^{\ttt}
\left(\begin{array}{cccc}
1 & 0   & 0 & 0\\
0 & p-1 & 0 & 0\\
0 & 0   & 1 & 0\\
0 & 0   & 0 & 1
\end{array}\right)\bfy=
np^2-p-np(l-k)+p(v-u),\]
or the negative of that for the counter-clockwise orientation
of~$L$. For the different choices of $(k,l)$ and $(u,v)$,
these $2(p-1)(n-1)$ rotation numbers are pairwise
distinct and distinguish the Legendrian realisations.

The number of $2$-handles is $p+2$, so the handlebody $X$
described by the Kirby diagram in Figure~\ref{figure:p-np-1}
has $\chi(X)=p+3$. From the Kirby moves in \cite[Figure~7]{geon18}
one reads off that $\sigma(X)=1-(p+1)=-p$.
The solution of the equation $M'\bfx=\vrot$ is
\[ \bfx=\left(\begin{array}{c}
np^2-p-np(l-k)+p(v-u)\\
-np+1+n(l-k)-(v-u)\\
-np+n(l-k)-(v-u)\\
p-(l-k)
\end{array}\right).\]
We then compute
\[ c^2=\bfx^{\ttt}\left(\begin{array}{cccc}
1 & 0   & 0 & 0\\
0 & p-1 & 0 & 0\\
0 & 0   & 1 & 0\\
0 & 0   & 0 & 1
\end{array}\right)\vrot=n(p-l+k)^2+2(p-l+k)(v-u)-p.\]
Putting all this information into the formula (\ref{eqn:d_3}) for $d_3$
gives the value claimed in Proposition~\ref{prop:p-np-1}.
Again we observe that the $d_3$-invariant never takes the value~$-1/2$.
\begin{ack}
This project was initiated during a \emph{Research in Pairs} stay
at the Mathematisches Forschungsinstitut Oberwolfach, and it was continued
during a couple of visits by S.O.\ to the Universit\"at zu K\"oln,
supported by a Turkish Academy of Sciences T\"UBA--GEBIP
and by the SFB/TRR 191 `Symplectic Structures in Geometry, Algebra
and Dynamics', funded by the DFG. We thank Sebastian Durst and Marc Kegel
for useful conversations and comments on a preliminary version of this paper.
We also thank the two referees for constructive comments.
\end{ack}

\end{document}